\documentclass[english,reqno,11pt]{amsart}
\usepackage{amsmath, amsthm, amsfonts}
\usepackage{hyperref} 
\usepackage{hyperref}
\hypersetup{
	colorlinks=true,
		citecolor=blue!60!black,
		linkcolor=red!60!black,
		urlcolor=green!40!black,
		filecolor=yellow!50!black,
	breaklinks=true,
	pdfpagemode=UseNone,
	bookmarksopen=false,
}
\usepackage{amsmath,amsthm,amssymb}
\usepackage{amscd,indentfirst,epsfig}
\usepackage{latexsym}
\usepackage{times}
\usepackage{enumerate}
\usepackage{mathrsfs}
\usepackage{amsopn}
\usepackage{amsmath}
\usepackage{amssymb,mathtools}
\usepackage{amsfonts,bm}
\usepackage{amsbsy,amsmath}
\usepackage{amscd}
\usepackage{xcolor}

\usepackage[T1]{fontenc}
\usepackage{amsthm}
\usepackage{microtype}

\begin{document}
\makeatletter
\def \leq {\leqslant}
\def \le {\leq}
\def \geq {\geqslant}
\def\e{\alpha}
\def \ge {\geq}
\def\tend{\rightarrow}
\def\R{\mathbb R}
\def\S{\mathbb S}
\def\Z{\mathbb Z}
\def\N{\mathbb N}
\def\C{\mathscr{C}}
\def\D{\mathscr{D}}
\def\T{\mathcal T}
\def\E{\mathscr E}
\def\s{\sigma}
\def\k{\kappa}
\def \a{\beta}
 \def\be{\beta}
\def\t{\theta}
\def\l{\ell}
\def \M {\mathcal{M}}
\def\L{\Lambda}
\def\O{\Omega}
\def\g{\gamma}
\def \ds {\displaystyle}
\def\G{\Gamma}
\def \mM {\mathfrak{m}}
\def\f{\varphi}
\def\o{\omega} 
\def \d {\mathrm{d}}
\def \lm {\bm{m}}
\def \dD {\mathcal{R}} 
\def\U{\Upsilon}
\def\z{\zeta}
\def \Q {\mathcal{Q}}
\def\over{\bm}
\def\b{\backslash}
\def\fet{f_{\ast}}
\def \get{g_{\ast}}
\def\fprim{f^{\prime}}
\def\fprimet{f^{\prime}_\ast }
\def\vet{v_{\ast}}
\def  \pa {\partial}
\def \var {\dd}
\def\vprim{v^{\prime}}
\def\vprimet{v^{\prime}_{\ast}}
\def\grad{\nabla}

\def\Log{\textrm{Log }}
\newtheorem{theo}{Theorem}[section]
\newtheorem{prop}[theo]{Proposition}
\newtheorem{cor}[theo]{Corollary}
\newtheorem{lem}[theo]{Lemma}
\newtheorem{hyp}[theo]{Assumptions}
\newtheorem{defi}[theo]{Definition}
\newtheorem{nb}[theo]{Remark}
\def \vr {\vartheta}
\def \ll {\lambda}
\def \up {\Upsilon}
\def \kk {\kappa}
\def \dd {\bm{\varepsilon}}

\def \ind {\mathbf{1}}
\numberwithin{equation}{section}

\title[Time-asymptotic rate for the Fermi-Dirac-Fokker-Planck equation]{An entropy-level method for quantitative time-asymptotic in the Fokker-Planck model for fermions -- appearance of non saturation}
\author{R. Alonso}
\address{College of Science and Engineering, Division of Sciences, HBKU, Education City, Doha, Qatar.} \email{ralonso@hbku.edu.qa}
 
\author{B.  Lods}
\address{Universit\`{a} degli
Studi di Torino \& Collegio Carlo Alberto, Department of Economics, Social Sciences, Applied Mathematics and Statistics ``ESOMAS'', Corso Unione Sovietica, 218/bis, 10134 Torino, Italy.}\email{bertrand.lods@unito.it}

\begin{abstract}
In this contribution we consider the Fokker-Planck equation for fermions, commonly known as the Fermi-Dirac-Fokker-Planck equation.  This equation is a correction of the classical Fokker-Planck model that incorporates the Pauli exclusion principle.  We present a novel approach that combines three ingredients: (1) stability of solutions, (2) quantitative convergence for a dense set of initial data via a robust entropy method, and (3) a De Giorgi level-set iteration, to establish quantitative time-asymptotic thermalisation for general initial data.  In particular, we prove quantitative convergence for arbitrarily cold initial data.
\end{abstract}

\maketitle


%
\section{Introduction}

The Fokker-Planck equation for fermions, commonly known as the Fermi-Dirac-Fokker-Planck equation (FDFP in the sequel), is a crucial model in describing the dynamics of particles obeying Fermi-Dirac statistics. This equation represents a significant refinement of the classical Fokker-Planck model by incorporating the Pauli exclusion principle, a fundamental aspect of quantum mechanics governing the behavior of fermions. The FDFP equation arises in various physical contexts, including plasma physics, condensed matter physics, and nuclear physics, where understanding the collective behavior of fermions is essential. In the spatially homogeneous setting, the FDFP equation reads
\begin{equation}\label{LFD}\begin{cases}
\partial_{t} f(t,v) &=  {\grad}_v \cdot \left(\grad_{v} f   +\,v f (1- \dd_0 f) \right),\qquad (t,v)\in (0,\infty)\times\mathbb{R}^{3}\,,\\
\\
 f(0,v)&=f^{\rm in}(v)\,, \qquad v \in \R^{3},\end{cases}
\end{equation}
with the common shorthand $f:= f(t,v)$, and quantum exclusion parameter
$$\dd_0:= \frac{(2\pi\hslash)^{3}}{m^{3}\alpha} >0$$
and where $f(t,v)$ denotes the density of fermions with velocity $v \in \R^{3}$ at time $t >0.$ Here, $\hslash$ denotes the reduced Planck constant, $m$ the particle mass, and $\alpha >0$ is the statistical weight of the particle species, i.e. the number of independent quantum states in which the particle can possess the same internal energy. We refer to \cite[Chapter 17]{chapman} as well as the pioneering works \cite{Nord,UU} for the justification of kinetic theory for Fermi-Dirac particles.\medskip

The significance of the FDFP equation \eqref{LFD} lies in its ability to capture the essential physics of fermionic systems while remaining analytically tractable. It serves as a bridge between microscopic kinetic descriptions, such as the Boltzmann-Fermi-Dirac equation \cite{borsoni, chapman, Lu, luwennberg}, and macroscopic fluid dynamics models. While more fundamental, the Boltzmann-Fermi-Dirac equation is often challenging to analyze directly. The FDFP equation offers a simplified, yet accurate, representation that allows for detailed investigation of phenomena like thermalization, transport, and stability. Moreover, the FDFP shares analogies with models in different contexts, such as a generalization of the Keller-Segel model for chemotaxis \cite{chava}, particularly in scenarios involving blow-up and high densities where finite size effects become important. For such a model, at late stages of blow-up (corresponding to large values of density),   finite size effects and stickiness have to be  taken
into account and Keller-Segel equation can be replaced with a model of the form
$$\partial_t \varrho(t,x)=D \Delta_x \varrho(t,x) -\chi\nabla_x\left(\varrho(t,x)\left(1-\frac{\varrho(t,x)}{\sigma_0}\right)\nabla_x c\right)$$
which enforces the constraint $\varrho \leq \sigma_0$ on the maximum concentration of bacteria in physical
space.  The FDFP equation corresponds to a harmonic potential $c(x)=-\frac{1}{2}|x|^2$ and parameters $D=\chi=1$, $\sigma_0=\dd_0^{-1}$.

Equation \eqref{LFD} was studied in \cite{CLR}, where the authors considered several aspects of the model, including the Cauchy problem, stability, and the rate of convergence toward equilibrium. The main quantitative result is \cite[Theorem 3.5]{CLR}, which gives an explicit exponential rate. It applies to large initial data; however, the initial datum must be controlled by a Fermi-Dirac distribution introduced below in (1.2). This makes the convergence result in \cite[Theorem 3.5]{CLR} \emph{conditional}. In particular, it applies only to strongly decaying initial data that lie away from the \emph{saturation level} $\frac{1}{\dd_0}$. This condition excludes very cold initial states for which $f^{\rm in}=\frac{1}{\dd_0}$ on large parts of the domain.

In this work, we extend \cite[Theorem 3.5]{CLR} to general initial data, even in the extreme scenario where all the mass is concentrated in the region $\{f^{\rm in} = \frac{1}{\dd_0}\}$. We present a novel approach to analyzing the FDFP equation, combining three key elements:
\begin{enumerate}
\item[(1)] the stability of solutions in $L^{1}$-topology,
\item[(2)] the quantitative convergence of a dense initial data set via a robust entropy theory,
\item[(3)] An iteration level argument reminiscent of De Giorgi oscillation lemma \cite{degiorgi} method where the combination of an \textit{a priori} pointwise bound and a ``supporting'' integral estimation on a solution's level can be used to \emph{improve} the pointwise property of such solution. We also refer to \cite{vasseur} (in particular \cite[Prop. 9]{vasseur} and related comments) for similar questions. 
\end{enumerate}
This combination allows us to demonstrate quantitative time-asymptotic thermalization for general initial data. Importantly, we can prove quantitative convergence even in cases of arbitrarily cold initial data, a regime where the quantum effects of the exclusion principle are dominant.
 
The long-term behavior of solutions to \eqref{LFD} is significantly influenced by the Fermi-Dirac distribution, defined for any $\beta > 0$ as
\begin{equation}\label{FDD}
\M_{\beta}(v):=\frac{1}{\dd_0+\beta\exp\left(\frac{1}{2}|v|^{2}\right)}<\frac{1}{\dd_0}, \qquad v \in \R^{3}\,.
\end{equation}
This distribution represents the unique class of steady-state solutions to Eq. \eqref{LFD}. The parameter $\beta > 0$ is intrinsically linked to the dynamics through the conserved mass of the initial distribution. 
In fact, the parameter $\beta>0$ is uniquely chosen by the equation dynamics through the initial distribution's mass, which is conserved in the equation flow.  Thus, we assume  
\begin{subequations}\label{IC0}
\begin{equation}\label{IC}
\int_{\R^{3}}f^{\rm in}(v)\d v=\varrho_{\rm in} \in(0,\infty) \qquad\text{and}\qquad \int_{\R^{3}}f^{\rm in}(v)v\d v=0.
\end{equation}
As a result, there is a unique $\beta_{0}=\beta(\varrho_{\rm in}) >0$ such that
\begin{equation}\label{FDIC}
\int_{\R^{3}}\M_{\beta_{0}}(v)\d v=\varrho_{\rm in}.
\end{equation}
In the sequel, we additionally assume that
\begin{equation}\label{ICA}
\int_{\R^{3}}f^{\rm in}(v)|v|^{2}\d v=E_{\rm in} >0 \qquad\text{and}\qquad \mathcal{H}(f^{\rm in}) >0\,,
\end{equation}
\end{subequations}
where for any suitable function $0 \leq g \leq \dd_0^{-1}$ we introduce the Fermi-Dirac entropy as 
\begin{equation}\label{FDEntropy}
\mathcal{H}(g)=\frac{1}{2}\int_{\R^{3}}g(v)|v|^{2}\d v + \int_{\R^{3}}\left[g(v)\log g(v)+\frac{1}{\dd_0}\big(1-\dd_0 g(v)\big)\log\big(1-\dd_0 g(v)\big)\right]\d v\,.
\end{equation}
Associated to this entropy there is a Fermi-Dirac entropy production, namely,
\begin{equation}\label{FDEntropyP}
\mathscr{D}(g) = \int_{\mathbb{R}^3}g(1-\dd_0 g)\Big| v + \nabla_{v}\log\Big(\frac{g}{1-\dd_0 g}\Big)\Big|^2 \d v\geq0\,,
\end{equation}
so that
\begin{equation}\label{FDE-EP}
\frac{\d}{\d t}\mathcal{H}(f(t)) + \mathscr{D}(f(t))=0\,.
\end{equation}
Due to the multiplicative factor $1-\dd_0 f\geq0$ in \eqref{FDEntropyP}, the quantitative long-time asymptotic analysis of equation \eqref{LFD} is more challenging than in the classical Fokker-Planck model; certainly in regimes where $f\simeq \frac{1}{\dd_0}$ occurs in large parts of the domain such as the important case of cold plasmas.  Thus, we seek to prove the appearance of a uniform \textit{saturation gap} 
\begin{equation}\label{eq:satur}
\kappa_{0}(t)=\inf_{v\in \R^{3}}\left(1-\dd_0\,f(t,v)\right) > 0, \qquad \forall\, t \geq T_{*}\,,
\end{equation}
for an explicitly computable time $T_*>0$ depending only on the initial data.  The saturation gap can be interpreted as the positive distance, in the $L^{\infty}$ metric, appearing between the lower energy levels of the Fermi gas and the actual state of the gas due to the warming produced by collisions. The relevance of an estimate of the type \eqref{eq:satur}, in particular when $\inf_{t} \kappa_0(t) >0$, for the long-time behavior of kinetic equations for Fermi-Dirac particles is a well-known feature of such equations (see \cite{liulu,borsoni} in the context of the Boltzmann-Fermi-Dirac equation and \cite{ABL,Alonso1,Alonso2} for the Landau-Fermi-Dirac equation).

In the reference \cite{CLR} the saturation gap has been proved assuming  
$$f^{\rm in}\leq \M_{\beta_*}<1$$ with $\beta_*>0$ as small as desired, which is interpreted as a relative warm condition on the Fermi gas.  In this case, such condition has been shown to propagate along the flow 
$$f(t)\leq\M_{\beta_*}, \qquad t \geq0\,,$$ proving the uniform propagation of the saturation gap.  Under such consideration, the convergence to the correct Fermi-Dirac distribution $\M_{\beta_0}$ follows by a classical entropy-entropy production argument \cite{carrillo} using the Fermi-Dirac entropy \eqref{FDEntropy}. 

In this document, we use this result in our favour using a type of density argument.  Indeed, we can conveniently approximate any initial datum with one satisfying the aforementioned control.  Then, with the knowledge that the equation \eqref{LFD} is stable in $L^{1}$, it is possible to show that general solutions to \eqref{LFD} are essentially controlled by $\M_{\beta_*}$ in $L^{1}$ up to small error.  Crucially, the size of the error is explicitly controlled in the approximation process. Finally, we use a De Giorgi $L^{1}-L^{\infty}$ type iteration level method to prove a saturation gap from which the quantification of the time-asymptotic follows.

Let us state the main results of the manuscript starting with the appearance of the saturation gap where, in all the sequel, we introduce the notations 
$$L^q_m(\R^3)=\left\{f\::\:\R^3\to \R\;,\;\|f\|_{L^{q}_m}:=\left(\int_{\R^3}\left(1+|v|\right)^{mq}\big|f(v)\big|^{q}\,\,\d v\right)^{\frac{1}{q}}  < \infty\right\}\,$$
for any $q \geq 1, m \geq0.$ For $m=0,$ we simply write $\|\cdot\|_{L^q_0}=\|\cdot\|_{L^q}.$

\begin{theo}[\textit{\textbf{Saturation gap}}]\label{theo:satgap}
Let  $p >3$, and $f^{\rm in} \in L^{1}_{p}(\R^{3})$ be such that $0\leq  f^{\rm in} \leq  {\frac{1}{\dd_{0}}}$   and let $f=f(t,v)$ be the solution to the Cauchy problem \eqref{LFD}.  Then, there exist an explicit $\delta_0=\delta_0( f^{\rm in} )$ and a time threshold $T_{\star}=T_{\star}( f^{\rm in} ) {>1}$ such that
\begin{equation*}
\kappa_{0}=\inf_{t\geq T_\star}\inf_{v\in \R^{3}}\left(1-\dd_0 f(t,v)\right) \geq \delta_0>0\,.
\end{equation*}
The dependence is explicit and only through $f^{\rm in}$ mass and energy.
\end{theo}
Theorem \ref{theo:satgap} leads to quantitative convergence for general initial data.
\begin{theo}[\textit{\textbf{Quantitative convergence}}]\label{theo:quantitative}
Let $p >3,$ and $f^{\rm in} \in L^{1}_{p}(\R^{3})$ be such that $0\leq  f^{\rm in} \leq  {\frac{1}{\dd_{0}}}$   and let $f=f(t,v)$ be the solution to the Cauchy problem \eqref{LFD}.  Then,
$$\mathcal{H}(f(t))-\mathcal{H}(\M_{\beta_{0}}) \leq \left(\mathcal{H}(f^{\rm in})-\mathcal{H}(\M_{\beta_{0}})\right)\exp\left(-2\kappa_0 (t - T_\star)\right)\qquad t\geq T_\star,$$
and, as a consequence of Csisz\'{a}r--Kullback inequality for fermions,
$$\left\|f(t)-\M_{\beta_{0}}\right\|_{L^1} \leq \sqrt{2\varrho_{\rm in}\left(\mathcal{H}(f^{\rm in})-\mathcal{H}(\M_{\beta_{0}})\right)}\,\exp\left(-\kappa_0\,(t-T_\star)\right) \qquad t \geq T_\star,$$
where $\M_{\beta_{0}}$ is the corresponding Fermi-Dirac statistic associated to $f^{{\rm in}}$, and $\kappa_0$ and $T_\star$ given in Theorem \ref{theo:satgap}.
\end{theo}
\begin{proof}
The argument follows \cite[Theorem 3.5]{CLR}; the central point is the inequality
\begin{align*}
\mathscr{D}(f) = &\int_{\mathbb{R}^3}f(1-\dd_0 f)\Big| v + \nabla_{v}\log\Big(\frac{f}{1-\dd_0 f}\Big)\Big|^2{\rm d}v\\ &\geq \kappa_0\int_{\mathbb{R}^3}f\Big| v + \nabla_{v}\log\Big(\frac{f}{1-\dd_0 f}\Big)\Big|^{2}{\rm d}v \geq 2\,\kappa_0\,\big( \mathcal{H}(f) - \mathcal{H}(\M_{\beta_0})\big)\qquad t\geq T_\star\,,
\end{align*}
which follows from Theorem \ref{theo:satgap} and the generalised logarithmic Sobolev inequality \cite[Theorem 17]{carrillo}, as in \cite{CLR}. Gr\"{o}nwall's lemma applied to \eqref{FDE-EP}, together with the monotonicity $\mathcal{H}(f(T_*))\leq \mathcal{H}(f^{{\rm in}})$, gives the result. For the $L^1$ convergence, the Csisz\'{a}r--Kullback inequality for fermions can be found in \cite[Theorem 3]{luwennberg} (see also \cite[Remark p. 386]{Lu} and \cite{AMTU}). The version in \cite{luwennberg,Lu} concerns a slightly different entropy, not involving the kinetic energy, because the associated steady state is assumed to share both mass and kinetic energy with the solution. In the Fokker-Planck case studied here, kinetic energy is not conserved, but entropy \eqref{FDEntropy} takes this into account and the proofs of \cite{Lu,luwennberg} adapt directly.\end{proof}

We assume in the sequel, without loss of generality and for simplicity, that
$$\dd_0=1.$$  Indeed, note that if $f$ solves the equation \eqref{LFD}, then $F=\dd_0f$ solves the Fermi-Dirac-Fokker-Planck equation with $\dd_0=1$, namely,
\begin{equation*}
\begin{cases}
\partial_{t} F(t,v) \!\!\!\!\!&=  {\grad}_v \cdot \left(\grad_{v} F   +\,v F (1- F) \right),\qquad (t,v)\in (0,\infty)\times\mathbb{R}^{3}\,, \\
\;F(0,v)&=F^{{\rm in}}(v)=\dd_0f^{\rm in}(v)\,.\end{cases}
\end{equation*}
\subsection*{Acknowledgments.} B. L. gratefully acknowledges the financial support from the Italian Ministry of Education, University and Research (MIUR), Dipartimenti di Eccellenza grant 2022-2027, as well as the support from the de Castro Statistics Initiative, Collegio Carlo Alberto (Torino).

\section{Preliminary results} 
In this section we recall some important results regarding solutions to \eqref{LFD} which can be found in the reference \cite{CLR}.  Let us introduce the spaces that are relevant in the discussion.   For any $T >0$, $p >3$, define
$$\Upsilon:=L^{\infty}(\R^{3}) \cap L^{1}_{1}(\R^{3}) \cap L^{p}_{1}(\R^{3}) \qquad\text{and}\qquad \Upsilon_{T}:=\mathcal{C}([0,T],\Upsilon)\,,$$
where $\Upsilon$ is endowed with the norm 
$$\|g\|_{\Upsilon}:=\max\left(\|g\|_{L^\infty},\|\,g\|_{L^{1}_1},\| g\|_{L^{p}_1}\right), \qquad g \in \Upsilon\,,$$
while $\Upsilon_{T}$ is endowed with the corresponding natural norm 
$$\|g\|_{\Upsilon_{T}}=\max_{0\leq t \leq T}\|g(t)\|_{\Upsilon}\,.$$
The well-posedness of the equation \eqref{LFD} is established in \cite[Theorem 2.11]{CLR}\,.   \begin{theo}\label{theo:cauchy}
Let $p >3,$ and $f^{\rm in} \in L^{1}_{p}(\R^{3})$ be such that $0\leq  f^{\rm in} \leq 1.$ Then, the Cauchy problem \eqref{LFD} with initial condition $f^{\rm in}$ has a unique global solution $f=f(t,v)$ such that
$$\int_{\R^{3}}f(t,v)\d v=\int_{\R^{3}}f^{\rm in}(v)\d v=\varrho_{\rm in}\qquad\text{and}\qquad 0 \leq f(t,v) \leq 1, \quad \forall\, t\geq 0.$$
In addition, for all $T >0$ we define the Banach space $X_T$ through the norm
$$\|g\|_{X_{T}}=\max\left\{\|g\|_{\Upsilon_{T}}\;;\;\sup_{0< t < T}\sqrt{e^{2t}-1}\,\|\nabla_{v} g(t)\|_{L^{p}_1}\;;\;\sup_{0< t < T}\sqrt{e^{2t}-1}\,\| g(t)\|_{L^{1}_1}\right\}\,.$$
Finally, one has the decay of the entropy
$$\mathcal{H}(f^{\rm in}) \geq\mathcal{H}(f(t)) \geq  \mathcal{H}(\M_{\beta_{0}}), \qquad \forall\, t \geq0.$$
\end{theo}
We notice the following important stability result, in the $L^{1}$-metric, for solutions to \eqref{LFD}.
\begin{prop}[\textit{\textbf{$L^{1}$-contraction and comparison principle}}]\label{prop:stable} Let $T >0$ and consider $f \,,\, g \in X_{T}$ two solutions to \eqref{LFD} with respective initial data $f^{\rm in} \,,\, g^{\rm in} \in \Upsilon$.
Then,
$$\left\|f (t)-g(t)\right\|_{L^1} \leq \left\|f^{\rm in}-g^{\rm in}\right\|_{L^1}\,, \qquad  0 \leq t < T\,.$$
Furthermore, if $\;0 \leq f^{\rm in} \leq g^{\rm in} \leq 1$ then $0\leq f(t,v) \leq g(t,v)$ for $\;0 \leq t \leq T$.
\end{prop}
\begin{nb}\label{nb:posi} Due to the conservation of mass, the $L^{1}$ stability result of Proposition \ref{prop:stable} can be reformulated as
$$\int_{\R^{3}}\left(f(t,v)-g(t,v)\right)^{+}\d v \leq \int_{\R^{3}}\left(f^{\rm in}(v)-g^{\rm in}(v)\right)^{+}\d v, \qquad  0 \leq t < T\,,$$
since $x^{+}=\frac{1}{2}(|x|+x)$ for any $x\in \R.$ 
\end{nb} 
Regarding the long-time asymptotic behaviour of solutions to \eqref{LFD}, the following quantitative result is available; see \cite[Theorem 3.5]{CLR}.
\begin{theo}\label{theo:CLR} Let $p >3,$  and $f^{\rm in} \in L^{1}_{p}(\R^{3})$ be such that $0\leq  f^{\rm in} \leq 1$.   Denote $\M_{\beta_{0}}$ the associated Fermi-Dirac distribution with the same mass $\varrho_{\rm in}>0$ as $f^{\rm in}$, and let $f=f(t,v)$ be the solution to the Cauchy problem \eqref{LFD}.  Furthermore, assume that there exists $\beta_{\star} >0$ such that
\begin{equation}\label{eq:UPP}
0 \leq f^{\rm in}  \leq \M_{\beta_{\star}} < 1\,.
\end{equation}
Then, with
$$\lambda=1-\|\M_{\beta_\star}\|_{L^\infty}=\frac{\beta_\star}{1+\beta_\star}>0,$$
$$\mathcal{H}(f(t))-\mathcal{H}(\M_{\beta_{0}}) \leq \left(\mathcal{H}(f^{\rm in})-\mathcal{H}(\M_{\beta_{0}})\right)\exp\left(-2\lambda t\right)\,,$$
and, as a consequence,
$$\left\|f(t)-\M_{\beta_{0}}\right\|_{L^1} \leq \sqrt{2\varrho_{\rm in}\left(\mathcal{H}(f^{\rm in})-\mathcal{H}(\M_{\beta_{0}})\right)}\,\exp\left(-\lambda\,t\right), \qquad t \geq0.$$
\end{theo}
Notice that time convergence to the Fermi-Dirac distribution without quantitative estimate of the rate can be deduced for any initial datum satisfying the assumptions of Theorem \ref{theo:cauchy} using a compactness argument. Namely, the following has been proven in \cite[Theorem 3.3]{CLR}.
 \begin{theo}[\textit{\textbf{Qualitative convergence}}]\label{theo:genCon} Let $p >3,$ and $f^{\rm in} \in L^{1}_{p}(\R^{3})$ be such that $0\leq  f^{\rm in} \leq 1$.  Denote $\M_{\beta_{0}}$ the associated Fermi-Dirac distribution with the same mass $\varrho_{\rm in}>0$ as $f^{\rm in}$.  Then 
$$\lim_{t\to\infty}\left\|f(t)-\M_{\beta_{0}}\right\|_{L^1}=0\,.$$
\end{theo}
\section{$L^1$ estimates by approximation argument}
Theorem \ref{theo:CLR} provides a time-exponential control for $f - \M_{\beta_{0}}$ in {$L^1$ distance through estimates of the entropy}.   Its proof is based on Proposition \ref{prop:stable} by using $g^{{\rm in}} = \M_{\beta_\star}$, thus, it is strongly restrictive on the initial data. {The subsequent} Theorem \ref{theo:L1imp} relaxes the condition on the initial datum providing instead an $L^{1}$-control up to a small error.  Such estimate will be combined later with a $L^{1}-L^{\infty}$ iteration level approach to prove the appearance of a saturation gap that leads to {the generalization of Theorem \ref{theo:CLR} to general initial distributions}. 
\begin{theo}\label{theo:L1imp}
Let $p >3,$  and $f^{\rm in} \in L^{1}_{p}(\R^{3})$ be such that $0\leq  f^{\rm in} \leq 1$. Denote $\M_{\beta_{0}}$ the associated Fermi-Dirac distribution with the same mass $\varrho_{\rm in}>0$ as $f^{\rm in}$ and let $f=f(t,v)$ be the solution to the Cauchy problem \eqref{LFD}.

Then, for any $\eta >0$ there exists an explicit $T_{\star}=T(\eta,\varrho_{\rm in},E_{\rm in}) >0$ such that
\begin{equation}\label{eq:quant}
\left\|f(t)-\M_{\beta_{0}}\right\|_{L^1}=2\int_{\R^{3}}\left(f(t,v)-\M_{\beta_{0}}(v)\right)^{+}\d v \leq \eta \qquad \forall \, t \geq T_{\star}.
\end{equation}
\end{theo}
The proof is divided into several lemmas. As mentioned earlier, the argument consists of three~ingredients: stability of the equation, Theorem \ref{theo:CLR}, and a suitable approximation of the initial datum $f^{\rm in}$ by a family of initial data satisfying assumption \eqref{eq:UPP}.  Let us start with the latter and construct a compactly supported initial datum $g_{0}=g_{R}$ with the same mass as $f^{\rm in}$.  For any $R >0$ and $v_{0} \in \R^{3}$ we set $B(v_{0},R):=\{v \in \R^{3}\;;\;|v-v_{0}| \leq R\}$ and simply write $B_{R}$ for $B(0,R)$.
\begin{lem}\label{lem:g0R} Let $R >0$ be given. Set
$$g_{0}(v)=g_{R}(v)=\begin{cases} f^{\rm in}(v) \qquad &\text{ for } v \in B_{R}\\
 \alpha_{R}  \qquad &\text{ for } v \in B_{\frac{3}{2}R}\setminus B_{R}\\
0 \qquad &\text{ else,} \end{cases}$$
where
$$\alpha_{R}=|B_{\frac{3}{2}R} \setminus B_{R}|^{-1}\left(\varrho_{\rm in}-\int_{\R^{3}\setminus B_{R}}f^{\rm in}(v)\d v\right)=\frac{1}{|B_{\frac{3}{2}R} \setminus B_{R}|}\int_{|v| >R}f^{\rm in}(v)\d v\geq0\,.$$
Then, $g_{R}(v) \in \Upsilon$ with $\mathcal{H}(g) >0$, $\|g_{R}\|_{L^1}=\varrho_{\rm in},$
and there exists an explicit $R_{0} >0$ such that 
$$0 \leq g_{R}(v) \leq \frac{1}{2} \qquad v \notin B_{R}\,,\quad \forall\, R >R_{0}.$$ 
\end{lem}
\begin{proof} The proof follows by direct inspection. Notice that
$$|B_{\frac{3}{2}R} \setminus B_{R}|=\frac{4}{3}\pi\left(\left(\frac{3}{2}R\right)^{3}-R^{3}\right)=\frac{19}{6}\pi\,R^{3}\,,$$
so that
$$\alpha_{R}=\frac{6}{19\pi\,R^{3}}\int_{|v|>R}f^{\rm in}(v)\d v \leq \frac{6E_{\rm in}}{19\pi\,R^{5}}, \qquad E_{\rm in}:=\int_{\R^{3}}f^{\rm in}(v)|v|^{2}\d v.$$
Consequently, one may take
$$R_0=\left(\frac{12E_{\rm in}}{19\pi}\right)^{1/5},$$
so that $\alpha_{R} \leq \frac{1}{2}$ for any $R \geq R_{0}$, which proves the result.
\end{proof}
Continuing with the same notation, we have the following lemma.
\begin{lem}\label{lem:gIN} Under the framework of Lemma \ref{lem:g0R}, for any $\varepsilon \in(0,1)$ set
$$g^{\rm in}_{\varepsilon,R}(v)=\begin{cases} \min\left(1-\varepsilon,\,g_{R}(v)\right) \qquad &\text{ for } v \in B_{\frac{3}{2}R}\\
\mu_R(\varepsilon) \qquad &\text{ for } v \in B_{2R} \setminus B_{\frac{3}{2}R}\\
0\qquad &\text{ else} ,\end{cases}$$
with $\mu_R(\varepsilon) \geq0$ chosen such that 
\begin{equation}\label{eq:g0REps}
\|g^{\rm in}_{\varepsilon,R}\|_{L^1}=\varrho_{\rm in}.\end{equation}
Then $g^{\rm in}_{\varepsilon,R} \in \Upsilon$, and there is an explicit $R_{1}(\varepsilon) > R_{0} >0$ such that
\begin{equation}\label{eq:g0Re1}
0 \leq g^{\rm in}_{\varepsilon,R} \leq 1-\varepsilon \qquad \forall\, R >R_{1}(\varepsilon).\end{equation}
Moreover, 
\begin{equation}\label{eq:L1fin}
 \int_{B_{R}}\left|f^{\rm in}(v)-g^{\rm in}_{\varepsilon,R}(v)\right|\d v \leq \varepsilon\,|B_{R}|.
\end{equation}
\end{lem}
\begin{proof} The proof follows, again, by direct inspection.  For $R$ sufficiently large, depending on $\varepsilon$, one has $\alpha_R\leq 1-\varepsilon$, and condition \eqref{eq:g0REps} reads
\begin{equation*}\begin{split}
\mu_R(\varepsilon)|B_{2R}\setminus B_{\frac{3}{2}R}| +\alpha_{R}|B_{\frac{3}{2}R}\setminus B_{R}|&=\varrho_{\rm in}-\int_{B_{R}}\min\left(1-\varepsilon, f^{\rm in}(v)\right)\d v\\
&\geq \varrho_{\rm in}-\int_{B_{R}} f^{\rm in}(v)\d v=\alpha_{R}|B_{\frac{3}{2}R}\setminus B_{R}|\,,
\end{split}\end{equation*}
where we used the definition of $\alpha_{R}$ in Lemma \ref{lem:g0R}. Therefore, $\mu_R(\varepsilon) \geq0$. Moreover,
$$\mu_R(\varepsilon)\leq |B_{2R}\setminus B_{\frac{3}{2}R}|^{-1}\varrho_{\rm in}= \frac{6}{37\pi\,R^{3}}\varrho_{\rm in}$$
and $\alpha_R\leq 6E_{\rm in}/(19\pi R^5)$. Thus, for instance, one may take
$$R_1(\varepsilon)=\max\left\{R_0,\left(\frac{6E_{\rm in}}{19\pi(1-\varepsilon)}\right)^{\!1/5},
\left(\frac{6\varrho_{\rm in}}{37\pi(1-\varepsilon)}\right)^{\!1/3}\right\},$$
from which \eqref{eq:g0Re1} follows. The proof of \eqref{eq:L1fin} is straightforward since on $B_{R}$ it holds that
$$g^{\rm in}_{\varepsilon,R}=\min(1-\varepsilon,g_{R})=\min(1-\varepsilon,f^{\rm in}).$$
From this, one checks that
$$
\int_{B_{R}}\left|f^{\rm in}(v)-g^{\rm in}_{\varepsilon,R}(v)\right|\d v=\int_{B_{R}}\left(f^{\rm in}(v)-(1-\varepsilon)\right)^{+}\d v \leq \varepsilon |B_{R}|
$$
since $f^{\rm in} \leq 1$ so that $\left(f^{\rm in}(v)-(1-\varepsilon)\right)^{+} \leq 1-(1-\varepsilon)=\varepsilon.$
\end{proof}
Now, the associated Fermi-Dirac solution $g_{\varepsilon,R}(t)$ to the initial datum $g^{\rm in}_{\varepsilon,R}$ satisfies the following time-asymptotic convergence.  
\begin{lem}\label{lem:gte} Under the framework of Lemma \ref{lem:gIN}, given $\varepsilon \in(0,1)$ and $R >R_{1}(\varepsilon)$, let $g_{\varepsilon,R}(t)=g_{\varepsilon,R}(t,v)$ be the unique solution to \eqref{LFD} with initial datum $g^{\rm in}_{\varepsilon,R}.$ Then,
$$\|g_{\varepsilon,R}(t)-\M_{\beta_{0}}\|_{L^1} \leq C_R(\varepsilon)\exp\left(-\lambda_{R}(\varepsilon) \, t\right) \qquad \forall \, t \geq0\,,$$
where $\M_{\beta_{0}}$ is the Fermi-Dirac distribution with mass $\varrho_{\rm in}.$
\end{lem}
\begin{proof} The proof is a simple consequence of Theorem \ref{theo:CLR} since, being compactly supported and bounded by $1-\varepsilon$, the initial datum $g^{\rm in}_{\varepsilon,R}$ is such that 
\begin{equation}\label{eq:gMRE}
0 \leq g_{\varepsilon,R}^{\rm in}(v) \leq \M_{\beta_{R}(\varepsilon)}(v) < 1, \qquad v \in \R^{3}\,,
\end{equation}
for $\beta_{R}(\varepsilon) >0$. Indeed, since $g_{\varepsilon,R}^{\rm in}(v)=0$ for $|v| >2R,$ inequality \eqref{eq:gMRE} holds for any $\beta_{R}(\varepsilon) >0.$ Now, since $g_{\varepsilon,R}^{\rm in}(v) \leq 1-\varepsilon$ for any $|v| \leq 2R,$ it is enough to choose $\beta_{R}(\varepsilon)$ such that 
$$1+\beta_{R}(\varepsilon)\exp\left(2R^{2}\right) \leq \frac{1}{1-\varepsilon}\,,$$
that is, $\beta_{R}(\varepsilon)=\frac{\varepsilon}{1-\varepsilon}\exp\left(-2R^{2}\right)$. The proof of Theorem \ref{theo:CLR} gives the explicit choices
\begin{equation}\label{eq:explicit-constants}
\lambda_R(\varepsilon)=1-\|\M_{\beta_R(\varepsilon)}\|_{L^\infty}
=\frac{\beta_R(\varepsilon)}{1+\beta_R(\varepsilon)},
\qquad
C_R(\varepsilon)=\sqrt{2\varrho_{\rm in}\big(\mathcal H(g^{\rm in}_{\varepsilon,R})-\mathcal H(\M_{\beta_0})\big)}.
\end{equation}
Since $g^{\rm in}_{\varepsilon,R}$ is supported in $B_{2R}$ and has mass $\varrho_{\rm in}$, one may use the estimate
\begin{equation}\label{eq:explicit-CR-bound}
C_R(\varepsilon)\leq \overline C_R:=\sqrt{2\varrho_{\rm in}\big(2R^2\varrho_{\rm in}-\mathcal H(\M_{\beta_0})\big)}.
\end{equation}
\end{proof} 
\begin{nb} Notice that the aforementioned result is true for any $\varepsilon \in(0,1)$ but not necessarily for $\varepsilon=0.$ Indeed, the choice of $\beta_R(\varepsilon) >0$ in \eqref{eq:gMRE} is made possible for any $\varepsilon \in(0,1)$.  Since one cannot exclude $g^{\rm in}_{\varepsilon,R} \to 1$ as $\varepsilon \to 0$, such a choice is no longer possible for $\varepsilon=0$.
\end{nb}
With these lemmata at hand, let us prove Theorem \ref{theo:L1imp}. \begin{proof}[Proof of Theorem \ref{theo:L1imp}] We use the notation of Lemmas \ref{lem:gIN} and \ref{lem:gte}. For $\varepsilon \in(0,1)$ and $R>R_1(\varepsilon)$, we consider an initial datum $g_{\varepsilon,R}^{\rm in}$ as in Lemma \ref{lem:gIN} and the associated solution $g_{\varepsilon,R}(t)$ to \eqref{LFD}. For any $t\geq0,$ it holds
\begin{multline*}
\int_{\R^{3}}\left|f(t,v)-\M_{\beta_{0}}(v)\right|\d v\\
\leq \int_{\R^{3}}\left|f(t,v)-g_{\varepsilon,R}(t,v)\right|\d v + \int_{\R^{3}}\left|g_{\varepsilon,R}(t,v)-\M_{\beta_{0}}(v)\right|\d v \\
\leq \int_{\R^{3}}\left|f^{\rm in}(v)-g_{\varepsilon,R}^{\rm in}(v)\right|\d v + C_R(\varepsilon)\exp\left(-\lambda_{R}(\varepsilon)\,t\right)\end{multline*}
where we used Proposition \ref{prop:stable} and Lemma \ref{lem:gte} in the last estimate. We focus on the estimate of 
$$I_{\varepsilon,R}=\int_{\R^{3}}\left|f^{\rm in}(v)-g_{\varepsilon,R}^{\rm in}(v)\right|\d v=2\int_{\R^{3}}\left( f^{\rm in}(v)-g_{\varepsilon,R}^{\rm in}(v)\right)^{+}\d v$$
using that $f^{\rm in}$ and $g^{\rm in}_{\varepsilon,R}$ share the same mass. We split the integral as
\begin{multline*}\frac12 I_{\varepsilon,R} = \int_{B_{R}}\left( f^{\rm in}(v)-g_{\varepsilon,R}^{\rm in}(v)\right)^{+}\d v + \int_{\R^{3}\setminus B_{R}}\left( f^{\rm in}(v)-g_{\varepsilon,R}^{\rm in}(v)\right)^{+}\d v\\
\leq \int_{B_{R}}\left|f^{\rm in}(v)-g_{\varepsilon,R}^{\rm in}(v)\right|\d v + \int_{\R^{3}\setminus B_{R}}f^{\rm in}(v)\d v\\
\leq \varepsilon|B_{R}| + \frac{1}{R^{2}}\int_{\R^{3}\setminus B_{R}}f^{\rm in}(v)|v|^{2}\d v
\leq\frac{4\pi}{3}\varepsilon\,R^{3}+\frac{E_{\rm in}}{R^{2}}.
\end{multline*} 
Combining these estimates gives
\begin{equation}\label{eq:fSta}
\left\|f(t)-\M_{\beta_{0}}\right\|_{L^1}
\leq  \frac{8\pi}{3}\varepsilon\,R^{3}+ {\frac{2E_{\rm in}}{R^{2}} }+ C_R(\varepsilon)\exp\left(-\lambda_{R}(\varepsilon)\,t\right)\end{equation}
holds true for any $\varepsilon \in(0,1)$, $R > R_{1}(\varepsilon)$, and $t \geq0$. We now give fully explicit choices. For $\eta>0$, set
\begin{align*}
R&=\max\left\{1,\sqrt{\frac{6E_{\rm in}}{\eta}},
\left(\frac{12E_{\rm in}}{19\pi}\right)^{1/5},
\left(\frac{12\varrho_{\rm in}}{37\pi}\right)^{1/3}\right\},\\
\varepsilon&=\min\left\{\frac12,\frac{\eta}{8\pi R^3}\right\},
\qquad
T_\star=\frac{1}{\lambda_R(\varepsilon)}
\left[\log\left(\frac{3\overline C_R}{\eta}\right)\right]^+,
\end{align*}
where $\lambda_R(\varepsilon)$ and $\overline C_R$ are given by \eqref{eq:explicit-constants} and \eqref{eq:explicit-CR-bound}. These choices ensure $R\geq R_1(\varepsilon)$ and make each of the three terms in \eqref{eq:fSta} at most $\eta/3$. This proves \eqref{eq:quant} and the claimed explicit dependence.
\end{proof}
\begin{nb} Theorem \ref{theo:L1imp} induces a quantitative version of Theorem \ref{theo:genCon} in the sense that it gives an explicit time $T_{\star}>0$ after which the solution $f(t,v)$ lies in a $L^{1}$-ball of radius $\eta$ centered at $\M_{\beta_{0}}$.  It seems possible that combining Theorem \ref{theo:L1imp} with a fine spectral analysis in $L^{1}$-spaces one may deduce an explicit exponential rate of convergence for arbitrary initial datum $f^{\rm in}$ under the assumptions of Theorem \ref{theo:L1imp} providing a generalization of Theorem \ref{theo:genCon} to a broader class of initial data.  We do not follow this route, instead we prove the appearance of a saturation gap via an $L^{1}-L^{\infty}$ regularisation procedure.
\end{nb}
\section{A De Giorgi's type argument for saturation gap appearance}\label{sec:gio}
The scope of this section is to prove the  appearance of a saturation gap 
\begin{equation*}
\kappa_{0}(t)=\inf_{v\in \R^{3}}\left(1-f(t,v)\right) > 0 \qquad \forall \, t \geq T_{*}\,,
\end{equation*}
with $T_*>0$ an \emph{explicit} time depending only on the initial data.  The $L^{1}-L^{\infty}$ regularisation procedure via iteration of the solution's levels allows to deduce from estimate \eqref{eq:quant} a pointwise estimate.  In this section, we consider an initial datum $f^{\rm in}$ satisfying the assumptions of Theorem \ref{theo:cauchy} and denote with $f(t,\cdot)$ the associated solution to \eqref{LFD} as constructed in Theorem \ref{theo:cauchy}.  {In particular, we assume $f^{\rm in} \in L^{1}_{p}(\R^{3})$ with $p >3$ in the sequel to ensure Theorem \ref{theo:cauchy} to apply.} We also assume that \eqref{IC0} holds true.
\subsection{Level-set energy iteration method}
The level sets of a solution $f=f(t,v)$ to equation \eqref{LFD} are described by
$$ {F_{\ell}(t,v)=(f(t,v)-\l)\,\ind_{\{f(t)\geq\l\}}=(f(t,v)-\l)^+,\qquad  \l \geq0\,.}$$
The equation has the natural energy functional
\begin{equation}\label{eq:EnergY}
\mathscr{E}_{\ell}(T_{1},T_{2}) : = 2\sup_{t \in [T_{1},T_{2})}\Bigg[\frac{1}{2} \left\|F_{\ell}(t)\right\|_{L^2}^{2} +  \int_{T_{1}}^{t} \left\|\nabla_v\,F_{\ell}(s) \right\|_{L^2}^{2}\d s\Bigg]\,.\end{equation}
 We have the following lemma.
\begin{lem}\label{lem:fl+} 
 Let $f(t,\cdot)$ be a weak solution to \eqref{LFD}.  For any $\l \in (0,1)$ it holds that
\begin{equation}\label{eq:13Ric0}
\frac{1}{2}\frac{\d}{\d t}\|F_{\ell}(t)\|_{L^2}^{2} +  \left\|\nabla_vF_{\ell}(t)\right\|_{L^2}^{2}  
\leq  \frac{3}{2} (1-2\ell)  \|F_{\ell}(t)\|_{L^2}^{2}+3 \l\left(1-\l\right)\|F_{\ell}(t)\|_{L^1}.
\end{equation}
In particular, for any $\l \in \left(\frac{1}{2},1\right)$ it follows that
\begin{equation}\label{eq:13Ric}
\frac{1}{2}\frac{\d}{\d t}\|F_{\ell}(t)\|_{L^2}^{2} + \left\|\nabla_vF_{\ell}(t)\right\|_{L^2}^{2}  
\leq 3\l\left(1-\l\right) \|F_{\ell}(t)\|_{L^1}.
\end{equation}
\end{lem} 
\begin{proof} In the set $\{f \geq \l\}$ it follows that $\partial_{t}F_{\l}=\partial_{t}f$ and $\nabla_vF_{\l}=\nabla_vf$.  Therefore, multiplying equation \eqref{LFD} with $F_{\l}$ and integrating over $\R^{3}$ it holds that
\begin{multline*}
\frac{1}{2}\dfrac{\d}{\d t}\|F_{\l}(t)\|_{L^2}^{2}=\int_{\R^{3}}\partial_{t}f(t,v)F_{\l}(t,v)\d v
=\int_{\R^{3}}F_{\l}(t,v) {\grad}_v \cdot \left(\grad_{v} f  + \,v f (1- f) \right)\d v\\
=-\int_{\R^{3}}\nabla_{v} F_{\l}(t,v)\cdot \nabla_{v}f(t,v)\d v - \int_{\R^{3}}v \cdot \nabla_{v}F_{\l}(t,v) f(t,v)(1-f(t,v))\d v\\
=-\int_{\R^{3}}|\nabla_{v}F_{\l}(t,v)|^{2}\d v - \int_{\R^{3}}v \cdot \nabla_{v}F_{\l}(t,v) f(t,v)(1-f(t,v))\d v\,.\end{multline*}
Now, on the set $\{f \geq \l\}$ the following identities hold 
$$f(1-f)=F_{\l}(1-F_{\l})+\l(1-\l)-2\l\,F_{\l}\,,$$ 
and
$$f(1-f)\nabla_{v} F_{\l}=\nabla_{v}\left(\tfrac{1}{2}F_{\l}^{2} -\tfrac{1}{3}F_{\l}^{3} \right) +\l(1-\l)\nabla_{v}F_{\l} -\l\nabla_{v}F^{2}_{\l}\,.$$
One concludes that
\begin{multline*}
\frac{1}{2}\dfrac{\d}{\d t}\|F_{\l}(t)\|_{L^2}^{2}+\|\nabla_{v}F_{\l}(t)\|_{L^2}^{2}=-\int_{\R^{3}}v\cdot \nabla_{v}\left(\tfrac{1}{2}F_{\l}^{2}(t,v)-\tfrac{1}{3}F_{\l}^{3}(t,v)\right)\d v \\- \l(1-\l) \int_{\R^{3}}v\cdot \nabla_{v}F_{\l}(t,v)\d v -\l\int_{\R^{3}}v\cdot \nabla_{v}F^{2}_{\l}(t,v)\d v\,\\
=3\int_{\R^{3}}\left(\tfrac{1}{2}F_{\l}^{2}(t,v)-\tfrac{1}{3}F_{\l}^{3}(t,v)\right)\d v +3\l(1-\l) \int_{\R^{3}} F_{\l}(t,v)\d v +3\l\int_{\R^{3}} F^{2}_{\l}(t,v)\d v\,,
 \end{multline*} 
which gives \eqref{eq:13Ric0} neglecting the cubic term.  Then, we deduce \eqref{eq:13Ric} using that $1-2\l \leq 0$ for $\l \geq \frac{1}{2}$.
\end{proof} 
\begin{lem} \label{lem:cflfl}  Let $f(t,\cdot)$ be a weak solution to equation \eqref{LFD}. Then, for any  $0 \leq T_{1}< T_{2} \leq T_{3}$ it holds that
\begin{align}\label{functionalineq}
 \mathscr{E}_{\ell}(T_2,T_3) \leq \frac{1}{T_2 - T_1}\int^{T_2}_{T_1}&\|F_{\ell}(t) \|_{L^2}^{2}\d t \nonumber\\&+6\l(1-\l)\int_{T_{1}}^{T_{3}} \|F_{\ell}(s)\|_{L^1}\d s\,,\qquad \forall\,\ell \in \left[\tfrac{1}{2},1\right]\,.
\end{align} 
\end{lem}
\begin{proof}  Fix $0 \leq T_{1} < T_{2} \leq T_{3}$ and $\frac{1}{2} \leq \l \leq 1$ . Integrating inequality \eqref{eq:13Ric} over $(t_{1},t_{2})$ with $T_{1}\leq t_{1} \leq T_{2} \leq t_{2}\leq T_{3}$ we deduce that
\begin{equation*}
 \|F_{\ell}(t_{2})\|_{L^2}^{2}+ 2\int_{T_{2}}^{t_{2}} \left\|\nabla_vF_{\ell}(s)\right\|_{L^2}^{2}\d s \leq  \left\|F_{\ell}(t_{1})\right\|_{L^2}^{2}+6\l(1-\l)\int_{T_{1}}^{t_{2}} \|F_{\ell}(s)\|_{L^1}\d s\,.
 \end{equation*}
Taking the supremum over $t_{2} \in [T_{2},T_{3}]$ it follows that
\begin{equation*}
\mathscr{E}_{\l}(T_{2},T_{3}) \leq  \left\|F_{\ell}(t_{1})\right\|_{L^2}^{2}+6\l(1- \l)\int_{T_{1}}^{T_{3}} \|F_{\ell}(s)\|_{L^1}\d s\, \qquad t_{1} \in [T_{1},T_{2}].
\end{equation*}
Integrating with respect to $t_{1} \in [T_{1},T_{2}]$ we readily deduce \eqref{functionalineq}. 
\end{proof}
Let us control the term involving the $L^{1}$-norm in \eqref{functionalineq} using the improved $L^{1}$-convergence result given in Theorem \ref{theo:L1imp}.  {Since the mass of $f^{\rm in}$ is positive, the associated Fermi-Dirac distribution $\M_{\beta_{0}}$ is such that $\beta_0 >0$} and 
\begin{equation}\label{eq:delta0}
\delta_{0}:=1-\|\M_{\beta_{0}}\|_{L^\infty} >0.\end{equation}
 {Up to reducing $\delta_0$, there is no loss of generality in assuming}
$$\l_{0}:=1-\delta_{0} \geq \frac{1}{2}.$$
\begin{prop}\label{prop:L1Ener} Fix $\eta>0$ and let $f(t,\cdot)$ be a solution to \eqref{LFD} associated with an initial datum $f^{\rm in}$ with mass $\varrho_{\rm in}$. There exists an explicit $T_{\star}(\eta)>0$ {(depending also on $\varrho_{\rm in}$ and $E_{\rm in}$)} such that
\begin{equation}\label{eq:L1Est}
6\l(1- \l)\int_{T_{1}}^{T_{3}} \|F_{\ell}(s)\|_{L^1}\d s \leq \eta\,(T_{3}-T_{1})\end{equation}
holds true for any $T_{\star}(\eta) \leq T_{1} \leq T_{3}$ and $\l \in (\l_{0},1)$. As a consequence, 
\begin{equation}\label{eq:E00} 
\mathscr{E}_{\ell}(T_{2},T_{3}) \leq \eta \left(1+ T_{3}- T_{2} \right)\,
\end{equation} 
is satisfied for any $\l \in (\l_{0},1)$ and  {$T_{\star}(\eta)  \leq T_1 < T_{2} \leq T_{3}$. }
\end{prop}
\begin{proof} Since $\M_{\beta_{0}} \leq 1-\delta_{0}$, for any $\l \geq \l_{0}$ it holds that $\{f \geq \l\} \subset \{f \geq \M_{\beta_{0}}\}$. Therefore,
$$\|F_{\l}(t)\|_{L^1}=\int_{\R^{3}}(f(t,v)-\l)^{+}\d v \leq \int_{\R^{3}}\left(f(t,v)-\M_{\beta_{0}}\right)^{+}\d v.$$
According to Theorem \ref{theo:L1imp}, for any $\eta >0$ there exists $T_{\star}(\eta) >0$ such that
$$6(1-\l)\|F_{\l}(t)\|_{L^1} \leq \eta \qquad \forall\, t \geq T_{\star}(\eta)\,,$$
which gives \eqref{eq:L1Est}, recalling that $\l \leq 1$.  With regard to \eqref{eq:E00}, observe that, since $f(t) \leq 1$,
\begin{equation}\label{eq:fInd}
0 \leq F_{\ell}(t,v) \leq  (1-\l)\ind_{\{f(t) \geq \l\}}\end{equation} and, therefore,
$$\|F_{\ell}(t)\|_{L^2}^{2}=\int_{\{f(t) \geq \l\}}(F_{\ell}(t,v))^{2}\d v \leq  (1-\l) \int_{\R^{3}}F_{\l}(t,v)\d v \leq  \eta$$
from the previous point provided that $t \geq T_{\star}(\eta)$. As a consequence, 
\begin{equation}\label{eq:avera} 
\frac{1}{T_{2}-T_{1}}\int_{T_{1}}^{T_{2}}\|F_{\ell}(s)\|_{L^2}^{2}\d s \leq  \eta\,, \end{equation}
for $T_{1} >T_{\star}(\eta)$.  Thus, using  \eqref{functionalineq} it follows that
$$\mathscr{E}_{\ell}(T_2,T_3) \leq \eta\left(1+T_{3}-T_{1}\right) \qquad \forall\, T_{\star} \leq T_{1} \leq T_{2} \leq T_{3}$$
and one gets \eqref{eq:E00} by letting $T_{1}\to T_{2}$.
\end{proof}
The fundamental result for the implementation of the level set iteration is given by the following proposition. 
\begin{prop}\label{main-energy-functional}  Let  $f(t,\cdot)$ be a  solution to \eqref{LFD}. Then, there exists an universal constant $C>0$  such that 
\begin{equation}\label{eq:ElT2T3q}
\mathscr{E}_{\l}(T_{2},T_{3}) \leq C(\l-k)^{-\frac{4}{3}}
\left(\frac{1}{T_{2}-T_{1}}+\frac{1-\l}{\l-k}\right)\mathscr{E}_{k}(T_1,T_3)^{\frac{5}{3}}\end{equation}
for any times $0 \leq T_{1} < T_{2} \leq T_{3}$ and levels {$0 \leq k < \l \in \left[\frac{1}{2},1\right]\,$}. 
\end{prop}
\begin{proof}
Fix $0 \leq T_{1} < T_{2} \leq T_{3}$ and estimate the terms in the right side of \eqref{functionalineq} as follows:  the term $\|F_{\ell}\|_{L^1}$ is controlled by $\|F_{\ell}\|_{L^2}$ using the key observation that if $0 \leq k <\ell$ it follows that 
$$0 \leq F_{\ell}\leq F_{k}  \qquad\text{and}\qquad \mathbf{1}_{\{f \geq \l\}} \leq \frac{F_{k}}{\l - k}\,,$$ so that, for any $\alpha \geq0$ it holds that
\begin{equation*}\label{eq:alphaf}
F_{\ell}(t) \leq F_\l(t) \left( \frac{F_{k}(t)}{\l-k}\right)^{\alpha} \leq \left(\l-k\right)^{-\alpha}\,\left(F_{k}(t)\right)^{1+\alpha} \qquad 0 \leq k < \l.
\end{equation*}
Applying this with $0 \leq k_{0} < \l$, where $k_{0}$ will be chosen below, we obtain
$$\|F_{\ell}(t)\|_{L^1} \leq  \left(\l-k_{0}\right)^{-1}\|F_{k_{0}}(t)\|_{L^2}^{2}$$
so that
\begin{equation*}\label{eq:mulcflfl}
6\l(1-\l)\int_{T_{1}}^{T_{3}} \|F_{\ell}(s)\|_{L^1}\d s \leq 6\frac{1-\l}{\l-k_{0}}\int_{T_{1}}^{T_{3}} \|F_{k_{0}}(s)\|_{L^2}^{2}\d s, \qquad k_{0} < \l\,.
\end{equation*}
Use now for this term and for the $L^{2}$-term in the right side of \eqref{functionalineq} the interpolation estimate derived in \cite[Lemma 4.2, Eq. (4.5)]{Alonso2}
\begin{equation}\label{eq:flLq}
\|F_{\ell}(t)\|_{L^2}^{2} \leq \frac{c_{0}}{(\l-k)^{\frac{4}{3}}}\, \|F_{k}(t)\|_{L^2}^{\frac{4}{3}}\,\left\|\nabla_{v} \,F_{k}(t)\right\|_{L^2}^{2}, \qquad 0 \leq k < \l\,,
\end{equation}
for a universal constant $c_{0} >0$.  We conclude that
$$ 
\frac{1}{T_2 - T_1}\int^{T_2}_{T_1}\| F_{\ell}(s)\|_{L^2}^{2}\,\d s   
\leq \frac{c_{0}}{(T_2 - T_1)(\ell - k)^{\frac{4}{3}}}
\int_{T_{1}}^{T_{2}} \|F_{k}(s)\|_{L^2}^{\frac{4}{3}}\left\|\nabla_vF_{k}(s)\right\|_{L^2}^{2}\d s
$$
and, recalling the definition of the energy functional, this results in 
\begin{equation*}\label{eq:fl2nu}
\frac{1}{T_2 - T_1}\int^{T_2}_{T_1}\| F_{\ell}(s)\|_{L^2}^{2}\,\d s 
\leq \frac{c_{1}}{(T_2 - T_1)(\ell - k)^{\frac{4}{3}}}\,\mathscr{E}_{k}(T_{1},T_{2})^{\frac{5}{3}}
\end{equation*}
for some universal $c_{1} >0$.  With regard to the second term in \eqref{functionalineq} we resort to estimate \eqref{eq:mulcflfl} with the choice $k_{0}=\frac{k+\l}{2}$ so that $\l-k_{0}=k_0 - k = \frac{\l-k}{2}$ to deduce, thanks to \eqref{eq:flLq}, that   
\begin{equation*}\begin{split} 
6\l(1-\l)\int_{T_{1}}^{T_{3}} &\|F_{\ell}(s)\|_{L^1}\d s \leq 6\frac{1-\l}{\l-k_{0}}\int_{T_{1}}^{T_{3}} \|F_{k_{0}}(s)\|_{L^2}^{2}\d s\\
&\leq \frac{6c_{0}(1-\l)}{\left(\frac{\l-k}{2}\right)^{1+\frac{4}{3}}}\int_{T_{1}}^{T_{3}}  \|F_{k}(s)\|_{L^2}^{ \frac{4}{3}}\left\|\nabla_vF_{k}(s)\right\|_{L^2}^{2}\d s\leq c_{2}\frac{(1-\l)}{(\l-k)^{\frac{7}{3}}}\mathscr{E}_{k}(T_{1},T_{3})^{ \frac{5}{3}}
\end{split}\end{equation*}
for some universal constant $c_{2} >0$. This concludes the proof.
\end{proof} 
\subsection{Proof of Theorem \ref{theo:satgap} -- Appearance of a saturation gap}  The proof follows three main steps.

\subsubsection{First step: Sequence $(\E_{n})_{n}$ of energies} We recall the definition of the initial gap and level,
$$\delta_{0}=1-\|\M_{\beta_{0}}\|_{L^\infty} >0 \qquad\text{and}\qquad \l_{0}=1-\delta_{0}.$$
Fix
$0<\Delta \leq T$ and define the following sequences of gaps, levels, and times
$$
\delta_n:=\frac{\delta_{0}}{2}\left(1+\frac{1}{2^{n}}\right),\qquad  k_n := 1-\delta_{n} , \quad \text{ and } \quad  t_{n}:=T-\frac{\Delta}{2}\left(1+\frac{1}{2^{n}}\right), \qquad n \in \N.
$$
Define the sequence of energies
\begin{equation}\label{defEFlevels}
\mathscr{E}_{n}:=\mathscr{E}_{k_n}(t_n,T)\,,\qquad n \in \N \,.
\end{equation}
Apply Proposition \ref{main-energy-functional} with the choices 
$$k=k_{n},\qquad \l=k_{n+1},  \qquad T_{1}=t_{n},\qquad T_{2}=t_{n+1}, \qquad T_{3}=T$$ so that
$$\l-k=\frac{\delta}{2^{n+2}} \qquad\text{and}\qquad T_{2}-T_{1}=\frac{\Delta}{2^{n+2}},$$
while 
$$\l_{0}\leq \l \leq 1 \qquad\text{and}\qquad 1-\l=1-k_{n+1}=\delta_{n+1} \leq \delta_{0}\,.$$
With these choices it follows from \eqref{eq:ElT2T3q} that 
\begin{align}\label{ControlEFlevels}
\mathscr{E}_{n+1} \leq C\left(\frac{\delta_{0}}{2^{n+2}}\right)^{-\frac{4}{3}} &\left(\frac{2^{n+2}}{\Delta} + 2^{n+2}\right)\,\mathscr{E}_{n}^{ \frac{5}{3}} \nonumber\\
&\leq C \delta_{0}^{-\frac{4}{3}} 2^{\frac{7}{3}(n+2)}\left(1+\frac{1}{\Delta} \right)\mathscr{E}_{n}^{\frac{5}{3}}, \qquad n \in\N .
\end{align}
\subsubsection{{Second step:} Upper barrier sequence} Introduce the sequence
$$\mathscr{U}_{n}:=\E_{0}\,Q^{-n},  \qquad n \in \N,$$
with $Q >1$ to be determined in such a way that, for  suitable choice of $\Delta$, the sequence $\left(\mathscr{U}_{n}\right)_{n}$ satisfies \eqref{ControlEFlevels} with the reverse inequality, that is,
$$\mathscr{U}_{n+1}  \geq C \delta_{0}^{-\frac{4}{3}} 2^{\frac{7}{3}(n+2)}\left(1+\frac{1}{\Delta} \right)\mathscr{U}_{n}^{ \frac{5}{3}}.$$
Such an inequality is equivalent to 
\begin{equation}\label{E0-constrain}
1 \geq C \delta_{0}^{-\frac{4}{3}} 2^{\frac{7}{3}(n+2)}Q^{1-\frac{2}{3}n}\E_{0}^{\frac{2}{3}}\left(1+\frac{1}{\Delta} \right)\,.\end{equation}
Choosing 
$Q \geq 2^{\frac{7}{2}}$ and setting $C_{0}=C4^{\frac{7}{3}}$, one sees that \eqref{E0-constrain} will be true provided that
\begin{equation}\label{eq:constraintE0}
1 \geq C_{0}Q \delta_{0}^{-\frac{4}{3}} \E_{0}^{\frac{2}{3}}\left(1+\frac{1}{\Delta} \right)\,,\end{equation} 
which is seen as a condition on  $\Delta$ and $\E_{0} : =\E_{k_{0}}\left(t_{0},T\right).$ At this stage, we recall that $k_{0}=1-\delta_{0}=\l_{0}$ and we resort to our main result regarding the long-time behaviour of solutions to \eqref{LFD}. Namely, one observes from Proposition \ref{prop:L1Ener} that for any $\eta >0$ it follows that
$$\E_{0}=\E_{k_{0}}\left(t_{0},T\right) \leq \eta\left(1+T-t_{0}\right)$$
as soon as $t_{0} \geq T_{\star}(\eta)$, where $T_{\star}(\eta)$ is the explicit time given in Proposition \ref{prop:L1Ener}. Since $t_{0}=T-\Delta$ for $T > T_{\star}(\eta)+1$ and $\Delta \leq 1$, then
$$\E_{0} \leq \eta(1+\Delta)\,.$$
Consequently, condition \eqref{eq:constraintE0} is met provided that
$$
1 \geq C_{0}Q \delta_{0}^{-\frac{4}{3}} \bigg(1+ \Delta \bigg)^{\frac{2}{3}}\left(1+\frac{1}{\Delta} \right)\eta^{\frac{2}{3}}.
$$
Therefore, choosing $\Delta=1$ and, then,
$$\eta \leq (C_{0}Q)^{-\frac{3}{2}}\delta_{0}^{2} 2^{-\frac{5}{2}}$$
lead to condition \eqref{eq:constraintE0}. 
\subsubsection{{Final step:} A comparison argument}  The sequences $\left(\E_{n}\right)_{n}$ and $\left(\mathscr{U}_{n}\right)_{n}$ satisfy \eqref{ControlEFlevels} with a reverse inequality while coinciding for $n=0$, therefore,  
$$0 \leq \E_{n} \leq \mathscr{U}_{n} \qquad \forall\, n \in \N\,.$$
Since $Q >1$, it holds that $0\leq \lim_{n}\E_{n}\leq \lim_{n}\mathscr{U}_{n} = 0$.   Furthermore, since the sequence of levels $(k_{n})_{n}$  and the sequence of times $(t_n)_{n}$ are such that
$$\lim_{n}k_{n}=\l_{\infty}:= \left(1-\frac{\delta_{0}}{2}\right), \qquad \lim_{n}t_{n}=T-\frac{\Delta}{2}=T-\frac{1}{2},$$
this implies that
$$
\sup_{\tau\in[T-\frac{1}{2},T]}\| F_{\l_{\infty}}(\tau)\|_{L^2}^{2}=0\,,\qquad T \geq T_\star(\eta) +1,
$$
that is
$$
\| f(\tau) \|_{L^\infty}\leq  \left(1-\frac{\delta_{0}}{2}\right) \quad \text{for any}\quad \tau\in\left[T-\frac{1}{2},T\right]\,,\quad T\geq T_\star(\eta)+1\,.
$$
Because $T\geq T_\star(\eta)+1$ is arbitrary, this leads to
$$
\| f(t) \|_{L^\infty}\leq  \left(1-\frac{\delta_{0}}{2}\right) \quad \text{for any}\quad t\geq T_\star(\eta)+1>0\,.
$$
This concludes the proof of Theorem \ref{theo:satgap}.

\end{document}